\documentclass[12pt]{amsart}
\usepackage{pstricks-add}
\usepackage{amscd,amsmath,amsthm,amssymb}
\usepackage[left]{lineno}
\usepackage{pstcol,pst-plot,pst-3d}
\usepackage{color}
\usepackage{pstricks}
\newpsstyle{fatline}{linewidth=1.5pt}
\newpsstyle{fyp}{fillstyle=solid,fillcolor=verylight}
\definecolor{verylight}{gray}{0.97}
\definecolor{light}{gray}{0.9}
\definecolor{medium}{gray}{0.85}
\definecolor{dark}{gray}{0.6}
\usepackage{hyperref}
\usepackage[normalem]{ulem}
\usepackage{comment}

\usepackage{tikz}
\usepackage{float}
\usepackage{multicol}

 \def\NZQ{\mathbb}               
 
 \def\QQ{{\NZQ Q}}
 \def\ZZ{{\NZQ Z}}

 \def\frk{\mathfrak}               
 
 \def\pp{{\frk p}}

 \def\mm{{\frk m}}

 \def\ab{{\mathbf a}}
 \def\bb{{\mathbf b}}
 \def\xb{{\mathbf x}}

 \def\cb{{\mathbf c}}

 \def\pb{{\mathbf p}}
 \def\qb{{\mathbf q}}
 \def\ub{{\mathbf u}}
 \def\mb{{\mathbf m}}
 \def\vb{{\mathbf v}}
  \def\sb{{\mathbf s}}
 \def\opn#1#2{\def#1{\operatorname{#2}}} 
 \opn\chara{char} \opn\length{\ell} \opn\pd{pd} \opn\rk{rk}
 \opn\projdim{proj\,dim} \opn\injdim{inj\,dim} 
 \opn\depth{depth} \opn\grade{grade} \opn\height{height}
 \opn\embdim{emb\,dim} \opn\codim{codim}

 \opn\Tr{Tr} 
 \opn\superheight{superheight}\opn\lcm{lcm}
 \opn\trdeg{tr\,deg}
 \opn\reg{reg} \opn\lreg{lreg} \opn\ini{in} \opn\lpd{lpd}
 \opn\size{size} \opn\sdepth{sdepth}
 \opn\link{link}\opn\fdepth{fdepth}\opn\lex{lex}
 \opn\type{type}
 \opn\gap{gap}
 \opn\arithdeg{arith-deg}
 \opn\astab{astab}
  \opn\dstab{dstab}
  \opn\bigheight{bigheight}
 \opn\div{div} \opn\Div{Div} \opn\cl{cl} \opn\Cl{Cl}
 \opn\Spec{Spec} \opn\Supp{Supp} \opn\supp{supp} \opn\Sing{Sing}
 \opn\Ass{Ass} \opn\Min{Min}\opn\Mon{Mon}
 \opn\Ann{Ann} \opn\Rad{Rad} \opn\Soc{Soc}
 \opn\Im{Im} \opn\Ker{Ker} \opn\Coker{Coker} \opn\Am{Am}
 \opn\Hom{Hom} \opn\Tor{Tor} \opn\Ext{Ext} \opn\End{End}
 \opn\Aut{Aut} \opn\id{id}
 
 \opn\nat{nat}
 \opn\pff{pf}
 \opn\Pf{Pf} \opn\GL{GL} \opn\SL{SL} \opn\mod{mod} \opn\ord{ord}
 \opn\Gin{Gin} \opn\Hilb{Hilb}\opn\sort{sort}
 \opn\PF{PF}\opn\Ap{Ap}
 \opn\mult{mult}
 \opn\aff{aff}
 \opn\relint{relint} \opn\st{st}
 \opn\lk{lk} \opn\cn{cn} \opn\core{core} \opn\vol{vol}  \opn\inp{inp} \opn\nilpot{nilpot}
 \opn\link{link} \opn\star{star}\opn\lex{lex}\opn\set{set}
 \opn\width{wd}
 \opn\Fr{F}
 \opn\QF{QF}
 \opn\G{G}
 \opn\type{type}\opn\res{res}
 \opn\conv{conv}
 \opn\gr{gr}
 
 \def\pot#1#2{#1[\kern-0.28ex[#2]\kern-0.28ex]}

 \opn\dirlim{\underrightarrow{\lim}}
 \opn\inivlim{\underleftarrow{\lim}}
 \let\sect=\cap
 \let\dirsum=\oplus
 \let\tensor=\otimes
 \let\iso=\cong

 \let\to=\rightarrow
 \let\To=\longrightarrow
 \def\Implies{\ifmmode\Longrightarrow \else
         \unskip${}\Longrightarrow{}$\ignorespaces\fi}
 \def\implies{\ifmmode\Rightarrow \else
         \unskip${}\Rightarrow{}$\ignorespaces\fi}
 \def\iff{\ifmmode\Longleftrightarrow \else
         \unskip${}\Longleftrightarrow{}$\ignorespaces\fi}

 \let\:=\colon
 \newtheorem{Theorem}{Theorem}[section]
 \newtheorem{Lemma}[Theorem]{Lemma}
 \newtheorem{Corollary}[Theorem]{Corollary}
 \newtheorem{Proposition}[Theorem]{Proposition}
 \newtheorem{Remark}[Theorem]{Remark}
 
 \newtheorem{Example}[Theorem]{Example}
 
 \newtheorem{Definition}[Theorem]{Definition}

  \newtheorem{Claim}[Theorem]{Claim}
 \let\epsilon\varepsilon
 \let\kappa=\varkappa
 \def\qed{\ifhmode\textqed\fi
       \ifmmode\ifinner\quad\qedsymbol\else\dispqed\fi\fi}
 \def\textqed{\unskip\nobreak\penalty50
        \hskip2em\hbox{}\nobreak\hfil\qedsymbol
        \parfillskip=0pt \finalhyphendemerits=0}
 \def\dispqed{\rlap{\qquad\qedsymbol}}
 
 \opn\dis{dis}
 \def\pnt{{\raise0.5mm\hbox{\large\bf.}}}
 
 \opn\Lex{Lex}

\newcommand{\dist}{\operatorname{dist}}

\newcommand{\bbZ}{\mathbb{Z}}

\newcommand{\bbR}{\mathbb{R}}

\newcommand{\m}{\mathfrak{m}}
\DeclareMathOperator{\tr}{{tr}}
\DeclareMathOperator{\rank}{{rank}}
\newcommand{\xxi}{\xb^{\mb}}
\newcommand{\xxip}{\xb^{\mb'}}
\newcommand{\xxipp}{\xb^{\mb''}}

\usepackage{fullpage}

\usepackage{setspace}

\usepackage{hyperref}

\usepackage{enumerate}

\usepackage{graphicx}

\usepackage[all,cmtip]{xy}

\begin{document}

\title{Measuring the non-Gorenstein locus of Hibi  rings and normal affine semigroup rings}
\author{J\"urgen Herzog$^{1}$}

\address{$^{1}$Fachbereich Mathematik, Universit\"at Duisburg-Essen, 
45117, Essen, Germany} \email{juergen.herzog@uni-essen.de}

\author{Fatemeh Mohammadi$^{2}$}
\address{$^{2}$School of Mathematics, University of Bristol, BS8 1TW, Bristol, UK}
\email{fatemeh.mohammadi@bristol.ac.uk}

\author{Janet Page$^{3}$}
\address{$^{3}$School of Mathematics, University of Bristol, Bristol, BS8 1TW, UK, and the Heilbronn Institute for Mathematical Research, Bristol, UK}
\email{janet.page@bristol.ac.uk}

\dedicatory{ }

\thanks{}




\maketitle

\setcounter{tocdepth}{1}

\noindent{\bf Abstract.} {The trace of the canonical module of a Cohen-Macaulay ring  describes its non-Gorenstein locus. We study the trace of the canonical module of a Segre product of algebras, and we apply our results to compute the non-Gorenstein locus of toric rings. We provide several sufficient and necessary conditions for Hibi rings and normal semigroup rings to be Gorenstein on the punctured spectrum.}
\section{Introduction}

Let $R$ be a local or graded Cohen-Macaulay ring which admits a canonical module $\omega_R$. The trace  of an $R$-module $M$, denoted  $\tr(M)$, is the sum of all ideals $\varphi(M)$, where the sum is taken over all $R$-module homomorphisms $\varphi \:\; M\to R$.  It is noticed in \cite{HHSTraceofthecanonical} that the non-Gorenstein locus of $R$ is the closed subset of $\Spec(R)$ which is given by the set of prime ideals containing $\tr(\omega_R)$.  It follows that the height of $\tr(\omega_R)$ is a good measure for the non-Gorenstein locus of $R$. For example, $R$ is Gorenstein on the punctured spectrum (the open subset $\Spec(R)\backslash\{\mm\}$ of $\Spec(R)$) if and only if $\tr(\omega_R)$ is primary to the (graded) maximal ideal $\mm$ of $R$. In this note, we study the trace of the canonical module of Segre products of algebras, Hibi rings, and normal affine semigroup rings.

\smallskip

\noindent{\bf 1.1. Hibi rings.} In 1987, Hibi \cite{HDisplat} introduced a class of algebras which nowadays are called Hibi rings. They are defined using finite posets and naturally appear in various algebraic and combinatorial contexts; see for example \cite{herzog2005distributive}, \cite{ene2011monomial}, \cite{Howe2005} and \cite{Kim2018}.  Hibi rings are toric $K$-algebras defined over a field $K$.  They are  normal Cohen-Macaulay domains and their defining ideal admits a quadratic Gr\"obner basis. Recently, more subtle properties of Hibi rings have been studied. For example, Miyazaki \cite{miyazaki2018almost} classified level and almost Gorenstein Hibi rings,  and Page \cite{page2019frobenius} studied the Frobenius complexity of Hibi rings.

The combinatorics of Hibi rings are governed by their defining posets. Given a finite poset $P$ and a field $K$, the Hibi ring associated to $P$ and $K$, which we denote by $K[P]$, is the $K$-algebra generated by the monomials associated to poset ideals of $P$, see Definition~\ref{def} for details. Therefore, it is natural to ask how algebraic properties of $K[P]$ are reflected by properties of the poset $P$. A classical result of Hibi \cite[Corollary 3.d]{HDisplat} says that $K[P]$ is Gorenstein if and only if $P$ is a pure poset, that is, all maximal chains of $P$ have the same length. In \cite{HHSTraceofthecanonical}, Herzog, Hibi, and Stamate call a ring as above nearly Gorenstein if $\tr(\omega_R) =\m$, and they classify all nearly Gorenstein Hibi rings. Indeed, they show that $K[P]$ is nearly Gorenstein if and only if all connected components $P_i$  of $P$ are pure and $|\rank P_i-\rank P_j|\leq 1$ for all $i$ and $j$.

One of the main results of this paper is that $K[P]$ is Gorenstein on the punctured spectrum if and only if each  connected component of $P$ is pure; see Theorem~\ref{thm:Janet} and its Corollary~\ref{punctured}. Theorem~\ref{thm:Janet} also shows that in this case the trace of $\omega_R$ is a power of the maximal ideal. This  property is no longer valid for general toric rings which are Gorenstein on the punctured spectrum, as we show in Example~\ref{contrast}. The proof of Theorem~\ref{thm:Janet} is based on the explicit combinatorial description of the canonical and anti-canonical modules of a Hibi ring and on Theorem~\ref{TraceForSegreProductOfGorensteinRings} in which the trace of the canonical module of a Segre product of Gorenstein rings is computed. More generally, it is shown in Theorem~\ref{tracesegre} that the trace of the canonical module for a Segre product of Cohen-Macaulay toric rings  can be computed, up to a high enough truncation, by the traces of the canonical modules of the factors of the Segre product. Together with Lemma~\ref{height}, this result allows us to compute the height of the trace ideal of the canonical module of the Segre product.

Theorem~\ref{thm:Janet} implies the surprising fact that if $P$ is a connected poset and $K[P]$ is not Gorenstein, then $\height(\tr(\omega_{K[P]}))<\dim K[P]$. In  other words, for a connected poset $P$,  $K[P]$ is Gorenstein if and only if it is Gorenstein on the punctured spectrum, see Corollary~\ref{inparticular}. Note that if  $P$ is connected, then  $K[P]$  is not a proper Segre product of Hibi rings. Thus, one may ask more generally whether a Cohen-Macaulay toric ring,  which is not a proper Segre product of toric rings, is Gorenstein if and only if it is Gorenstein on the punctured spectrum. In Example~\ref{not a Segre product}, we show that this is not the case.

By applying a result of Miyazaki \cite{Miyazaki},  it follows from our Theorem~\ref{thm:Janet} that if $K[P]$ is Gorenstein on the punctured spectrum, then it is also a level ring. For toric rings which are not Hibi rings, this is not the case, as we show in Example~\ref{contrast}. In fact, there exist connected posets $P$ with the property that the non-Gorenstein locus of $K[P]$ may have arbitrarily large dimension. Indeed, we show in Corollary~\ref{ab} that for any given integers $a,b$ with $4\leq a<b$,  there exists a poset $P$ with $\height(\tr(\omega_{K[P]}))=a$ and $\dim K[P]=b$.

\smallskip

\noindent{\bf 1.2. Normal semigroup rings.}
In the last section, we study normal simplicial semigroup rings. Given a field $K$, and a rational polyhedral cone $\sigma$, one can associate the semigroup ring $R=K[S_\sigma]$ generated by integral points $\sigma\cap\mathbb{Z}^n$.
We will focus on the case that $\sigma \subset \ZZ^n$ is simplicial, that is, that $\sigma$ is cone with $n$ extremal rays.  We will say these are spanned by the primitive integral vectors $\ab_1,\ldots,\ab_n\in \ZZ^n$, where by primitive we mean that the coordinates of each $\ab_i$ have a $\gcd$ of 1.
It follows from a famous theorem of Hochster \cite{hochster1972rings}, that $R$  is normal and Cohen-Macaulay. Danilov and Stanley (see for example \cite[Theorem 6.3.5]{BrunsHerzogBook}) showed that the canonical module $\omega_R$  is the ideal in $R$ whose $K$-basis consists of the monomials $\xb^\ab$ with $\ab$ in the relative interior of $\sigma$.

On the other hand, the cone $\sigma$ can also be described by its inner normal vectors $\ub_1,\ldots,\ub_n$.  The vectors  $\ub_i$ are the integral vectors whose coordinates have a $\gcd$ of $1$, which satisfy the property that $\ab\in\sigma$ if and only if $\langle \ub_i,\ab\rangle \geq 0$ for $i=1,\ldots,n$. In our situation, $\omega_R$ has a nice presentation.  Indeed, if we denote by $\tau$ the subset of $\ZZ^n$ given by $\ab$ such that $\langle \ub_i,\ab\rangle \geq 1$, then the $K$-basis of $\omega_R$ consists of the monomials $\xb^\ab$ with $\ab\in \tau\sect \ZZ^n$. In the case that $\sigma$ is a simplicial cone, $\tau$ is also a cone.  This observation is crucial for the rest of the section (and is in general not valid for non-simplicial cones).  In  Proposition~\ref{prop:cone}, we observe that $R$ is Gorenstein if and only if the cone point of $\tau$ is an integral vector, and we can compute this cone point explicitly.  Theorem~\ref{integralpointsonrays} provides a lower bound for the height of the trace of $\omega_R$, and shows that $R$ is Gorenstein on the punctured spectrum if and only if there exist integral points on every extremal ray of  $\tau$. We translate this property into numeric conditions involving the coordinates of the cone point $\bb$ and the coordinates  of the vectors $\ab_j$ on the extremal rays of $\sigma$, see Corollary~\ref{first condition} and Corollary~\ref{special case}. Alternatively, in Proposition~\ref{thm:integral points} we provide a necessary condition, in terms of  the matrix whose rows are the inner normal vectors of $\sigma$,  for the existence of integral points on the extremal rays of $\tau$. This allows us to show that certain normal simplicial semigroup rings are not Gorenstein on the punctured spectrum, as is demonstrated by Example~\ref{exam:Janet}.

\section{The trace of the canonical module for Segre products}
In this section we develop some of the algebraic  tools which will be used  in the next section, and introduce trace ideals.  We refer to  \cite{HHSTraceofthecanonical} for a more complete introduction of the trace of the canonical module.

For any $R$-module $M$, its \textbf{trace} denoted by $\tr_R(M)$, or $\tr(M)$ when there is no confusion, is the sum of the ideals $\varphi(M)$ for $\varphi \in \Hom_R(M,R)$.  Namely, we have:
$$
\tr(M) := \sum_{\varphi \in \Hom_R(M,R)} \varphi(M).
$$

We note if $M_1 \cong M_2$ then $\tr(M_1) = \tr(M_2)$, so while the canonical module $\omega_R$ is unique only up to isomorphism, its trace is unique.  We are particularly interested in studying $\tr(\omega_R)$, as this measures the non-Gorenstein locus of $R$.  Namely,

\begin{Lemma}[\cite{HHSTraceofthecanonical} Lemma 2.1]
Let $\pp \in \Spec(R)$. Then $R_\pp$ is not a Gorenstein ring if and only if
$$
\tr(\omega_R) \subset \pp.
$$
\end{Lemma}
When $I$ is an ideal of positive grade, its trace ideal is $\tr(I) = I \cdot I^{-1}$.  We will be studying $\tr(\omega_R)$ in the case that $R$ is a Cohen-Macaulay domain, so that $\omega_R$ is either isomorphic to $R$ (in the case that $R$ is Gorenstein), or can be identified with an ideal of grade 1.  Then we will use the fact that $\tr(\omega_R) = \omega_R \cdot \omega_R^{-1}$.

\medskip

Let $\omega_{R}$ be the canonical module of $R$. The {\bf $a$-invariant} of $R$ is defined as
\[
a(R)=-\alpha(\omega_{R}),
\]
where for a  finitely generated graded module $M$, we set $\alpha(M)=\min\{i\:\; M_i\neq 0\}$.

\smallskip

Let $R=R_1\sharp R_2\sharp \cdots \sharp R_m$ be the Segre product of standard graded Cohen-Macaulay toric $K$-algebras, each of dimension $\geq 2$.

\begin{Proposition}
\label{CM}
$R$ is Cohen-Macaulay  if $a(R_i)<0$ for $i=1,\ldots,m$. In this case,
\[
\omega_R= \omega_{R_1}\sharp \cdots \sharp \omega_{R_m}.
\]
\end{Proposition}

\begin{proof}
For $m=2$, this follows from \cite[Theorem (4.2.3)(ii) and Theorem (4.3.1)]{goto1978}. Now let $m>1$, and set $S=R_1\sharp R_2\sharp \cdots \sharp R_{m-1}$. Then $R=S\sharp R_m$. We induct on $m$ and so we assume $S$ is Cohen-Macaulay and $\omega_S=\omega_{R_1}\sharp \cdots \sharp \omega_{R_{m-1}}$. From \cite[Theorem (4.2.3)(i)]{goto1978} it follows by induction on $m$  that $\dim S= \sum_{i=1}^{m-1} \dim R_i-(m-2)>2$. We also see that
$$a(S)=-\alpha(\omega_S)= \max\{-\alpha(\omega_{R_i})\:\; i=1,\ldots,m-1\}$$ which implies that $a(S)<0$. Thus \cite[Theorem (4.2.3)(ii) and Theorem (4.3.1)]{goto1978} applied to $S\sharp R_m$ yields the desired conclusion.
\end{proof}

\begin{Example}{\em
Let $P=P_1+P_1+\cdots +P_m$ be a finite poset with connected components $P_i$. We denote by $K[P]$ the Hibi ring associated with $P$, see Definition~\ref{def}. Then $K[P]=K[P_1]\sharp\cdots \sharp K[P_m]$, $\dim K[P_i]=|P_i|+1$ and  $a(K[P_i])=-(\rank P_i+2)$.  Thus Proposition~\ref{CM} can be applied.}
\end{Example}
From now on, we will assume that all $R_i$ have negative $a$-invariant.

\smallskip

For  a standard graded Cohen-Macaulay $K$-algebra $R$ with the canonical module $\omega_R$, the graded module $\omega_R^{-1}=\Hom_R(\omega_R,R)$ is the anti-canonical module of $R$. The trace $\tr(\omega_R)$  of $\omega_R$ is the graded ideal  whose $k^{\rm th}$ graded component is generated by the elements $\varphi(g)$ with $\varphi\in (\omega_R^{-1})_i$, $g\in (\omega_R)_j$  and $i+j=k$. Thus
\[
\tr(\omega_R)_k=\sum_{i+j=k}(\omega_R^{-1})_i(\omega_R)_j.
\]

It follows from Proposition~\ref{CM} and  \cite[Theorem 2.6]{anti2017} that for the Segre product $R$ as above we have
\[
\omega_R^{-1}= \omega_{R_1}^{-1}\sharp \cdots \sharp \omega_{R_m}^{-1}.
\]
The action of $(\omega_R^{-1})_i$ on  $(\omega_R)_j$ is given by the action on the factors as follows:
\[
(\omega_R^{-1})_i(\omega_R)_j= (\omega_{R_1}^{-1})_i(\omega_{R_1})_j\tensor_K \cdots \tensor_K(\omega_{R_m}^{-1})_i(\omega_{R_m})_j.
\]
Then, for all $k$,
\begin{eqnarray}
\label{trace}
\tr(\omega_R)_k=\sum_{i+j=k}(\omega_{R_1}^{-1})_i(\omega_{R_1})_j\tensor_K \cdots \tensor_K(\omega_{R_m}^{-1})_i(\omega_{R_m})_j.
\end{eqnarray}

For a graded module $M$ and an integer $k$ we set $M_{\geq k}=\dirsum_{i\geq k}M_i$.
By using the above  description of the graded components of $\omega_R$ we obtain

\begin{Theorem}
\label{tracesegre}
$
\tr(\omega_R)_{\geq k}=\tr(\omega_{R_1})_{\geq k}\sharp \cdots \sharp \tr(\omega_{R_m})_{\geq k}
$
for $k\gg 0$.
\end{Theorem}

\begin{proof}
Let $R$ be a standard graded $K$-algebra with a graded maximal ideal $\mm$ and let $M$ be a finitely generated $R$-module.
We  set  $\beta(M)=\max\{i\:\; (M/\mm M)_i\neq 0\}$. Thus $\beta(M)$ is the highest degree of a generator in a minimal set of generators of $M$, and $M_k=M_l R_{k-l}$ for all $k\geq l\geq\beta(M)$.
Let
\[
b=\max_i\{\beta(\omega_{R_i}^{-1})\}+\max_i\{\beta(\omega_{R_i})\}.
\]
Let $k\geq b$,  $s=\max_i\{\beta(\omega_{R_i}^{-1})\}$ and $t= k-s$. Then $t\geq \max_i\{\beta(\omega_{R_i})\}$, since $k\geq b$. Now let $1\leq l\leq m$ be an integer and consider $(\omega_{R_l}^{-1})_i(\omega_{R_l})_j$ with $i+j=k$. If $i<s$, then $j>t$. Therefore, $(\omega_{R_l})_j=(\omega_{R_l})_tR_{j-t}$. Hence,
\[
(\omega_{R_l}^{-1})_i(\omega_{R_l})_j=((\omega_{R_l}^{-1})_iR_{j-t})(\omega_{R_l})_t\subset (\omega_{R_l}^{-1})_s(\omega_{R_l})_t.
\]
On the other hand, if $i\geq s$, then $(\omega_{R_l}^{-1})_i=(\omega_{R_l}^{-1})_sR_{i-s}$. Therefore,
\[
(\omega_{R_l}^{-1})_i(\omega_{R_l})_j=((\omega_{R_l}^{-1})_s(R_{i-s})(\omega_{R_l})_j)\subset (\omega_{R_l}^{-1})_s(\omega_{R_l})_t.
\]
Applying (\ref{trace}) we obtain
\[
\tr(\omega_R)_k=(\omega_{R_1}^{-1})_s(\omega_{R_1})_t\tensor_K \cdots \tensor_K(\omega_{R_m}^{-1})_s(\omega_{R_m})_t.
\]
By the choice of $s$ and $t$, $(\omega_{R_l}^{-1})_s(\omega_{R_l})_t=(\omega_{R_l}^{-1}\omega_{R_l})_k$ for all $l$. Thus,
$$\tr(\omega_R)_k= \tr(\omega_{R_1})_{k}\sharp \cdots \sharp \tr(\omega_{R_m})_{k}.$$ This yields  the desired conclusion.
\end{proof}

In particular, when each $R_i$ is Gorenstein, we get the following result, which is a slight generalization of Theorem 4.15 in \cite{HHSTraceofthecanonical}.

\begin{Theorem}\label{TraceForSegreProductOfGorensteinRings}
Let $R = R_1 \sharp \cdots \sharp R_m$ where $R_i$ is a Gorenstein standard graded $K$-algebra for each $i$, and we assume $a_i$ is the $a$-invariant of $R_i$, and we have $-a_1 \geq \dots \geq -a_m > 0$.  Then $\tr(\omega_R) = \m^{a_m - a_1}$.
\end{Theorem}
\begin{proof}
We have that $\omega_R = \omega_{R_1} \sharp \cdots \sharp \omega_{R_m}$ by Proposition \ref{CM}. Similarly, this gives $\omega_R^{-1} = \omega_{R_1}^{-1} \sharp \cdots \sharp \omega_{R_m}^{-1}$.
Since each $R_i$ is Gorenstein, we have $\omega_{R_i} \cong R_i(a_i)$ and $\omega_{R_i}^{-1} \cong R_i(-a_i)$, so that we can identify $\omega_R \cong R_1(a_1) \sharp \cdots \sharp R_m(a_m)$ and $\omega_R^{-1} \cong R_1(-a_1) \sharp \cdots \sharp R_m(-a_m)$.  In the notation of Theorem \ref{tracesegre}, we have that $\beta(\omega_R) = -a_1$ and $\beta(\omega_R^{-1}) = a_m$, since we have assumed $-a_1 \geq \dots \geq -a_m > 0$.  By Theorem \ref{tracesegre}, we have that if $k \geq a_m - a_1$, then
\[
\tr(\omega_R)_k= \tr(\omega_{R_1})_{k}\sharp \cdots \sharp \tr(\omega_{R_m})_{k}
\]
so that
\[
\tr(\omega_R)_{\geq (a_m - a_1)}= (R_1)_{\geq (a_m - a_1)}\sharp \cdots \sharp (R_m)_{\geq (a_m - a_1)}
\]
since $\tr(\omega_{R_i}) = R_i$ as each $R_i$ is Gorenstein.
On the other hand, if $k < a_m - a_1$, then for every $i + j = k$ we have either $i < -a_1$ or $j < a_m$ so that $(\omega_R^{-1})_i(\omega_R)_j = 0$ as it either contains the term $(\omega_{R_1})_i$ or the term $(\omega_{R_m}^{-1})_j$ and both are $0$.

Then  $\tr(\omega_R)= \tr(\omega_R)_{\geq (a_m - a_1)}$, and so we have $\tr(\omega_R) \cong \m^{a_m - a_1}$.
\end{proof}

In particular, we also recover the following result from \cite{HHSTraceofthecanonical}.

\begin{Corollary}
If $R = R_1 \sharp \cdots \sharp R_m$ where each $R_i$ is a Gorenstein standard graded $K$-algebra with the $a$-invariant $a_i$, then $R$ is nearly Gorenstein if and only if $|a_i - a_j| \leq 1$ for all $i$ and $j$. 
\end{Corollary}

To relate the Gorenstein-ness properties of $R$ and $R_i$ we need the following lemma.

\begin{Lemma}
\label{height}
Let $I_j\subset R_j$ be graded ideals. Then
\[
\height_R(I_1\sharp I_2\sharp \cdots \sharp I_m)=
\begin{cases}
\dim R, \text{if $\height I_j=\dim R_j$ for $j=1,\ldots, m$},\\
\min\{\height I_j\:\; j=1,\ldots,m\}, \text{otherwise}.
\end{cases}
\]
\end{Lemma}

\begin{proof}
Note that
\begin{eqnarray*}
I_1\sharp I_2\sharp \cdots \sharp I_m&=&(I_1\sharp R_2\sharp  R_3\sharp\cdots \sharp R_m)\sect (R_1\sharp I_2\sharp R_3\sharp\cdots \sharp R_m)\sect \ldots \\
&\sect & (R_1\sharp R_2\sharp \cdots \sharp  R_{m-1}\sharp I_m).
\end{eqnarray*}
This implies that
\[
\height_R(I_1\sharp I_2\sharp \cdots \sharp I_m) =\min\{\height_R(R_1\sharp\cdots \sharp I_j\sharp \cdots  \sharp R_m)\:\; j=1,\ldots,m\}.
\]
It therefore suffices to show that
\[
\height_R(R_1\sharp\cdots \sharp I_j\sharp \cdots  \sharp R_m)=
\begin{cases}
\dim R, \text{ if $\height I_j=\dim R_j$},\\
\height I_j, \text{ otherwise}.
\end{cases}
\]
We may assume $j=1$. We set $S=R_2\sharp \cdots \sharp R_m$. Then $R=R_1\sharp S$. First, suppose that $\height I_1=\dim R_1$.  Then $\dim_K(I_1\sharp S)_k= \dim_K(R_1\sharp S)_k =\dim_KR_k$ for $k\gg 0$. This shows that $\height(I_1\sharp S)=\dim R$.

Next we assume that $\height I_1<\dim R_1$. We denote by $P(M)$ the Hilbert polynomial of a finitely generated graded module. Then $P(M)\neq 0$, if $\dim M>0$ and $\dim M=\deg P(M)+1$.

We have
\begin{eqnarray*}
\dim_K( (R_1\sharp S)/(I_1\sharp S))_k &=&\dim_K(R_1\sharp S)_k- \dim_K(I_1\sharp S)_k\\
&=& (\dim_K(R_1)_k)(\dim_KS_k)- (\dim_K(I_1)_k)(\dim_KS_k)\\
&=& (\dim_K(R_1/I_1)_k)(\dim_KS_k).
\end{eqnarray*}
This shows that $P((R_1\sharp S)/(I_1\sharp S))= P(R_1/I_1)P(S)$. Our assumption implies that  $P(R_1/I_1)\neq 0$. Therefore,
\begin{eqnarray*}
\dim ((R_1\sharp S)/(I_1\sharp S))&=&\deg P(R_1/I_1)+ \deg  P(S) +1\\
&=& (\deg P(R_1/I_1)+1)+ (\deg P(S)+1)-1\\
&=& \dim R_1/I_1+\dim S-1.
\end{eqnarray*}
Similarly, $\dim(R_1\sharp S)=\dim R_1+ \dim S-1$. Thus,
\begin{eqnarray*}
\height(I_1\sharp S) &=&  \dim(R_1\sharp S)- \dim ((R_1\sharp S)/(I_1\sharp S))\\
&=& (\dim R_1+ \dim S-1)-(\dim R_1/I_1+\dim S-1)=\dim R_1-\dim R_1/I_1\\
&=& \height I_1.
\end{eqnarray*}

\end{proof}

\begin{Corollary}
\label{punctured}
$R$ is Gorenstein on the punctured spectrum if and only if this is the case for each $R_i$.
\end{Corollary}

\begin{proof}
Note that $\height(\tr(\omega_{R_i})_{\geq k})= \height(\tr(\omega_{R_i}))$. Now, the assertion follows from Theorem~\ref{tracesegre} and Lemma~\ref{height}.
\end{proof}

\section{The case of Hibi rings}
In this section, we first briefly introduce Hibi rings and some related notation. 

\smallskip

If $P$ is a poset, we will write $u \lessdot v$ or $v$ \textbf{covers} $u$ for $u,v \in P$ if $u \leq v$ and there is no $w \in P$ such that $u \leq w \leq v$.
For $u \leq v$, we will denote by $[u, v]$ the set of all elements $w \in P$ such that $u \leq w \leq v$.  We say a chain $v_0 \lessdot \dots \lessdot v_n$ has \textbf{length} $n$.  For any subset $S \subset P$, let \textbf{rank} $S$ denote the maximal length of any chain in $S$, so that $\rank[u,v]$ denotes the maximal length of a chain from $u$ to $v$.
Similarly, for $u \leq v$, let \textbf{dist}$(u,v)$ be the minimal length of any chain from $u$ to $v$.
It will often be useful to add a minimal element $-\infty$ and a maximal element $\infty$ to a poset $P$.  We denote this by $\hat{P}$.

We say $I \subset P$ is a \textbf{poset ideal} if for all $v \in I$ and $u \leq v$ we have $u \in I$.  We denote the set of poset ideals of $P$ by $\mathcal{I}(P)$.

\begin{Definition}\cite{HDisplat}\label{def}
Given a finite poset $P:= \{v_1,\ldots,v_n\}$ and a field $K$, the \textit{Hibi ring} associated to $P$ over a field $K$, which we denote by $K[P] \subset K[t,x_1,\dots,x_n]$, is the ring generated over $K$ by the monomials $tx_I := t\prod_{v_i \in I} x_i$ for every $I \in \mathcal{I}(P)$,
$$
K[P] := K[tx_I | I \in \mathcal{I}(P)].
$$
\end{Definition}

To compute $\tr(\omega_R)$ for a Hibi ring $R$, we will use the following description of the canonical and anti-canonical modules (see also Proposition~\ref{descriptionoftau} and Corollary~\ref{descriptionoftautilde} as we can also view Hibi rings as normal affine semigroup rings).

\begin{Proposition}\label{GeneratorsForCanonical}\cite{SHilbfunofgradedalg}
Let $R = K[P]$, and let $\mb: \hat{P} \rightarrow \mathbb{Z}$, with $\mb(\infty) = 0$.  Let $\xxi =  t^{\mb(-\infty)}\prod_{v_i \in P} x_i^{\mb(v_i)}$.  Then $\xxi \in \omega_R$ if and only if $\mb$ satisfies the following:
\begin{align}
\mb(v_i) &\geq \mb(v_j) + 1 \text{ for } v_i \lessdot v_j \in \hat{P}.
\end{align}
\end{Proposition}

Similarly, we can use this to compute a $K$-basis of $\omega_R^{-1}$, as follows.

\begin{Corollary}\label{GeneratorsForAnticanonical}
Let $R = K[P]$, and let $\mb: \hat{P} \rightarrow \mathbb{Z}$, with $\mb(\infty) = 0$.  As before, let $\xxi =  t^{\mb(-\infty)}\prod_{v_i \in P} x_i^{\mb(v_i)}$.  Then $\xxi \in \omega_R^{-1}$ if and only if $\mb$ satisfies the following:
\begin{align}
\mb(v_i) &\geq \mb(v_j) -1 \text{ for } v_i \lessdot v_j \in \hat{P}.
\end{align}
\end{Corollary}

We showed in Theorem \ref{TraceForSegreProductOfGorensteinRings} that when $R = R_1 \sharp \cdots \sharp R_m$, there exists some $\ell$ such that $\tr(\omega_R) = \m^\ell$.  For Hibi rings, this characterizes rings in which $\tr(\omega_R) \supset \m^N$ for some specific integer $N$ associated to the underlined poset.  Namely, we have the following:
\begin{Theorem}\label{thm:Janet}
Let $R =K[P]$ be a Hibi ring, and $P = P_1 + \dots + P_m$ where the $P_i$ are the connected components of $P$, and let $N = \max \{\rank P_i - \rank P_j\}$ for all $i$ and $j$.  Then the following are equivalent:

\begin{enumerate}
    \item[{\em (1)}] $P_i$ is pure for all $i$ (i.e. $K[P_i]$ is Gorenstein for all $i$)
    \item[{\em (2)}] $\tr(\omega_R) = \m^N$
    \item[{\em (3)}] $\tr(\omega_R) \supset \m^\ell$ for some $\ell\geq 0$.
\end{enumerate}
\end{Theorem}
\begin{proof}
We have $(1) \implies (2)$ from Theorem \ref{TraceForSegreProductOfGorensteinRings}, and clearly $(2) \implies (3)$, so it suffices to show $(3) \implies (1)$.  Suppose $P_i$ is not pure, so that there exists an element $v_j \in P_i$ such that $\rank([v_j, \infty]) \neq \dist(v_j,\infty)$ in $\hat{P_i}$.  Then note that by definition of $P$, this is also true in $\hat{P}$.  Let $a = \rank([v_j, \infty])$ and $b = \dist(v_j,\infty)$ in $\hat{P}$ (this is the same as if we were to define them in $\hat{P_i}$).  Then by Proposition \ref{GeneratorsForCanonical}, we have that for any $\xxi \in \omega_{R}$
\begin{equation*}
    \mb(v_j) \geq a
\end{equation*}
and for any $\xxip \in \omega_{R}^{-1}$ we have by Corollary \ref{GeneratorsForAnticanonical}
\begin{equation*}
    \mb'(v_j) \geq -b.
\end{equation*}
In particular if $\xxipp \in \tr(\omega_{R})$ we have that
\begin{equation*}
    \mb''(v_j) \geq a-b > 0
\end{equation*}
so that since $v_j \neq -\infty$ we know $t^\ell \notin \tr(\omega_{R})$ since a power of $x_j$ appears in every monomial in $\tr(\omega_{R})$.  Thus, we cannot have $\m^\ell \subset \tr(\omega_R)$ for any $\ell\geq 0$.
\end{proof}

In particular, we get the following.

\begin{Corollary}
\label{puncturedhibi}
Let $P$ be a finite poset with connected components $P_i$. Then $K[P]$ is Gorenstein on the punctured spectrum if and only if each $P_i$ is pure.
\end{Corollary}

\begin{Corollary}
\label{level}
If $K[P]$ is Gorenstein on the punctured spectrum, then $K[P]$ is level (i.e. all generators of $\omega_{K[P]}$ have the same degree).
\end{Corollary}

\begin{proof}
In \cite[Theorem~3.3]{Miyazaki} Miyazaki showed that $K[P]$ is level if for all $x\in P$ all chains in $\hat{P}$ ascending from $x$ have the same length. This is obviously the case if  all connected components of $P$ are pure.
\end{proof}

In the case that $R$ cannot be written as a Segre product of smaller Hibi rings (i.e. $P$ is connected), we also obtain the following result.

\begin{Corollary}
\label{inparticular}
If $P$ is connected then $K[P]$ is Gorenstein if and only if it is Gorenstein on the punctured spectrum.
\end{Corollary}

In the example below, we can see that if $P$ is connected but not pure, then $K[P]$ is not Gorenstein on the punctured spectrum, i.e., for each $\ell>0$ we have that $\tr(\omega_R)  \not\supset \m^\ell$.

\begin{Example}{\em
Consider the following poset $P$ and its corresponding Hibi ring $K[P]$.
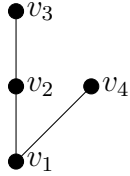
\begin{center}
\begin{tikzpicture}
\draw (1,-1) --(0,-2) -- (0,0);
\draw[fill] (0,0) circle [radius=0.1];
\draw[fill] (0,-1) circle [radius=0.1];
\draw[fill] (0,-2) circle [radius=0.1];
\draw[fill] (1,-1) circle [radius=0.1];

\node [right] at (0,0) {$v_3$};
\node [right] at (0,-1) {$v_2$};
\node [right] at (0,-2) {$v_1$};
\node [right] at (1,-1) {$v_4$};

\end{tikzpicture}
\end{center}
Then $\tr(\omega_{K[P]}) = (tx_1, tx_1x_4, tx_1x_2, tx_1x_2x_4, tx_1x_2x_3, tx_1x_2x_3x_4)$. Note that no power of $t$ belongs to $\tr(\omega_{K[P]})$. Moreover, $ \dim K[P]=5$ and $\height(\tr(\omega_{K[P]})) = 4 = \dim K[P] - 1$.}
\end{Example}

In the following we construct families of connected posets for which the dimension of the non-Gorenstein locus of the corresponding Hibi ring is as big as we want.

\smallskip

Given two (finite) posets $P_1$ and $P_2$ on disjoint sets, the ordinal sum of $P_1$ and $P_2$, denoted by $P_1\oplus P_2$, is defined to be the poset $P$ on the set $P_1 + P_2$ with order relation given as follows: if $p,q\in P_i$, and $p\leq q$ in $P_i$ for $i = 1$ or $2$, then $p\leq q$ in $P$, and if $p\in P_1$ and $q\in P_2$, then $p\leq q$. Note that in general $P_1\oplus P_2\not \iso P_2\oplus P_1$.

For a poset $Q$ we denote by $\bar{Q}$ the poset which is obtained from $Q$ by adding a minimal element to $Q$. The following result is observed in \cite[Page 434]{lattices}.

\begin{Proposition}
\label{ordinal}
Let $P_1$ and $P_2$ be posets on disjoint sets. Then $$K[P_1\oplus \bar{P_2}]=K[P_1]\tensor_K K[P_2].$$
\end{Proposition}

\begin{Corollary}
\label{traceordinal}
Let $P$ be the ordinal sum of $P_1$ and $\bar{P_2}$. Then $$\tr(\omega_{K[P]})=(\tr(\omega_{K[P_1]})K[P])\cdot(\tr(\omega_{K[P_2]})K[P])$$
\end{Corollary}

\begin{proof}
This follows from Proposition~\ref{ordinal} and \cite[Proposition 4.1]{HHSTraceofthecanonical}.
\end{proof}

Let $I\subset R$ be an ideal. Adopting  the convention that $\height I=-1$ if $I=R$, we obtain:

\begin{Corollary}
\label{heightordinal}
Let $P$ be the ordinal sum of $P_1$ and $P_2$. Then $$\height(\tr(\omega_{K[P]}))=\max\{\height(\tr(\omega_{K[P_1]})),\height(\tr(\omega_{K[P_2]}))\}.$$
\end{Corollary}

\begin{Corollary}
\label{ab}
Given integers $a$ and $b$  with $4\leq a< b$, there exists a connected poset $P$ such that $\height(\tr(\omega_{K[P]})) =a$ and $\dim K[P]=b$.
\end{Corollary}

\begin{proof}
Let $P=P_1\oplus \bar{P_2}$, where $P_1$ is a totally ordered poset with  $|P_1|=b-a-1$ and $P_2$ is a poset with connected components $Q_1$ and $Q_2$, where $Q_1$ is a totally ordered poset with $|Q_1|=a-2$ and $Q_2$ is the poset with $|Q_2|=1$. Then $|P_2|=a-1$, and $\dim K[P]=|P_1|+|P_2|+2=b$. By \cite{HDisplat}, $K[P_2]$ is not Gorenstein,  but on the other hand,
Corollary~\ref{puncturedhibi} implies that $K[P_2]$ is Gorenstein on the punctured spectrum. Therefore,
$\height(\tr(\omega_{K[P_2]}))=\dim K[P_2]=a$. Thus the desired conclusion follows from Corollary~\ref{heightordinal}.
\end{proof}

\section{The case of normal affine semigroup rings}

Again, we will briefly introduce the notation we will use throughout this section.  Throughout, we will denote $M := \bbZ^n$ for our monomial space, and for $\mb = (m_1, \ldots, m_n) \in M$, we will write $\xb^\mb$ to denote $x_1^{m_1} \cdots x_n^{m_n}$.  We will write $\gcd(\mb)$ to denote the $\gcd(m_1, \ldots, m_n)$.  If $\sigma$ is a cone in $M_\bbR = M \otimes \bbR$, we will write $S_\sigma$ for its corresponding semigroup, and we will denote:
$$
K[S_\sigma] := K[\xb^\mb \:\; \mb \in \sigma \cap M ].
$$
We will often decribe rational polyhedral cones $\sigma \subset M_\bbR$ in two ways, where by \textbf{rational} we mean that the extremal rays of $\sigma$ have integral generators, and by \textbf{polyhedral} we mean that $\sigma$ has finitely many extremal rays.  First, we can describe the extremal rays through $\ab_1, \ldots \ab_d \in M$ of $\sigma$. We always assume $\gcd(\ab_i) = 1$ if the corresponding ray has some integral point, as we can pick the first such integral point on the ray.  In this case, $\ab_i$ is called a \textbf{primitive} integral vector.  On the other hand, writing $N = M^\vee := \Hom_\bbZ(M, \bbZ) = \bbZ^n$ and $N_\bbR = N \otimes \bbR$, we can describe $\sigma$ by:
$$
\sigma := \{ \vb \in M_\bbR \:\;  \left<\vb, \ub_i\right> \geq 0 \text{ for some } \ub_1, \ldots, \ub_d \in N\},
$$
where again, we assume $\gcd(\ub_i) = 1$.
When we say $\sigma$ is \textbf{simplicial}, we mean that $d=n$, so that $\sigma$ has $n$ extremal rays.

We will often use the following description of $\omega_R$, which comes from the fact that $\omega_R$ is given by $\xb^\mb$ for $\mb$ in the interior of $\sigma$ (this is due to Danilov and Stanley; see for example \cite[Theorem 6.3.5]{BrunsHerzogBook}).  We note that the interior of $\sigma \cap M$ is equivalent to $\{\vb \:\;  \left<\vb, \ub_i \right> > 0 \} = \{\vb \:\;  \left<\vb, \ub_i \right> \geq 1 \}$, since we are considering only integral points in the interior of $\sigma$, and we have chosen $\ub_i$ to be primitive generators, namely $\gcd(\ub_i) = 1$.  Then we have the following:

\begin{Proposition}\label{descriptionoftau}
If $\sigma = \{ \vb \in M_\bbR \:\;  \left< \vb, \ub_i \right> \geq 0 \text{ for } \ub_1, \dots \ub_d \in N :=  \bbZ^n\}$ where $\gcd(\ub_i)=1$, then the canonical module of $K[S_\sigma]$ is given by:
$$
\omega_R = \left< \xb^\mb \:\;  \left<\mb, \ub_i\right> \geq 1  \text{ for } i = 1, \ldots d \right>.
$$
\end{Proposition}

We will denote $\tau : = \{ \vb \in M_\bbR \:\;  \left< \vb, \ub_i \right> \geq 1 \text{ for } \ub_1, \dots \ub_d \in N :=  \bbZ^n\}$, so that $\xb^{\mb} \in \omega_R$ if and only if $\mb \in \tau \cap M$.

Similarly, we note that this gives the following description of $\omega_R^{-1}$:

\begin{Corollary}\label{descriptionoftautilde}
If $\sigma = \{ \vb \in M_\bbR \:\;  \left< \vb, \ub_i \right> \geq 0 \text{ for } \ub_1, \dots \ub_d \in N :=  \bbZ^n\}$ where $\gcd(\ub_i) = 1$, then the anti-canonical module of $K[S_\sigma]$ is given by:
$$
\omega_R^{-1} = \left< \xb^\mb \:\;  \left<\mb, \ub_i\right> \geq -1  \text{ for } i = 1, \ldots d \right>.
$$
\end{Corollary}
\begin{proof}
Denote $\tilde{\tau} : =\{ \vb \:\;  \left<\vb, \ub_i\right> \geq -1  \text{ for } i = 1, \ldots d \}$ and suppose $\mb_1 \in \tilde{\tau} \cap M$.  Then for any $\mb_2 \in \tau \cap M$ (i.e. $\xb^{\mb_2} \in \omega_R$), 
we have that $\left<\mb_1 + \mb_2, \ub_i \right> = \left<\mb_1, \ub_i \right>  + \left<\mb_2, \ub_i \right> \geq 1 -1 = 0$.  Then $\xb^{\mb_1 + \mb_2} \in R$ so that $\xb^{\mb_1} \in \omega_R^{-1}$.  Now suppose $\xb^{\mb_1} \in \omega_R^{-1}$.  We need to show that $\left<\mb_1, \ub_i\right> \geq -1$ for all $i = 1, \ldots, d$.  Suppose for contradiction that $\left<\mb_1, \ub_i\right> <-1$ for some $\ub_i$.  We need the following claim (see also \cite{BGPolytopesringsandktheory}, page 216).

\begin{Claim}
There is some $\mb_2 \in M$ such that $\left< \mb_2, \ub_i\right> = 1$, and $\left< \mb_2, \ub_j\right> \geq 1$ for $i \neq j$.
\end{Claim}
\begin{proof}
We note that we can satisfy the first condition by iteratively applying the Euclidean algorithm.  Namely, we can find $m_1,m_2$ such that $m_1u_1 + m_2u_2 = \gcd(u_1,u_2)$, and then we can find $m_1', m_3$ such that $m_1'(m_1u_1 + m_2u_2) + m_3 u_3 = \gcd(u_1,u_2,u_3)$, and so on, so that we can find $\mb = (m_1, \ldots m_n)$ such that $\left<\mb, \ub_i\right> = \gcd(\ub_i) = 1$.

Now we show that we can also satisfy the other conditions simultaneously.
Suppose we have found such an $\mb$ by the method above, and let $J$ be the set of indices where $\left<\mb,\ub_j\right> \leq 0$ (so clearly $i \notin J$).  Let $c = \min\{\left<\mb,\ub_j\right>\:\; j\in J\}$ (so $c \leq 0$).  We claim there is an integral point $\sb \in \sigma \cap \bbZ^n$ such that $\left<\sb, \ub_i\right> = 0$ and $\left<\sb, \ub_j\right> \geq 1$ for all $j \in J$. If $\ab_j$ is a primitive generator of an extremal ray of $\sigma$, then by construction it is on the intersection of $n-1$ of the $(n-1)$-planes defined by $\ub_j$.
We consider a labelling such that $\ab_j$ is on the intersection of the $(n-1)$-planes defined by $\ub_1, \ldots, \hat{\ub_j}, \ldots, \ub_n$, so that $\left< \ab_l, \ub_j \right> = 0$ for $l \neq j$ and $\left<\ab_l, \ub_l\right> > 0$, i.e. $\left<\ab_l, \ub_l\right> \geq 1$.
Then let $\sb = \sum_{j \in J} \ab_j$.  Then $\left< \sb, \ub_j \right> \geq 1$ for $j \in J$ and $0$ otherwise.  Let $\mb_2 = \mb + (1-c)\sb$.  Then we have:
$$
\left<\mb_2,\ub_i\right> = \left<\mb ,\ub_i \right> +\left<(1-c)\sb,\ub_i \right> = 1 + 0
$$
and for all $j \in J$, we have:
$$
\left<\mb_2,\ub_j \right> = \left< \mb ,\ub_j \right> +\left<(1-c)\sb,\ub_j\right> \geq c + (1-c) = 1
$$
and for all other $j$, we have:
$$
\left<\mb_2,\ub_j\right> = \left<\mb, \ub_j \right> + \left<(1-c)\sb,\ub_j \right> \geq 1 + (1-c) > 1
$$
so that $\mb_2$ satisfies the desired conditions.
\end{proof}

By construction (and Proposition~\ref{descriptionoftau}), we have that $\xb^{\mb_2} \in \omega_R$, and further $\left<\mb_1 + \mb_2, \ub_i \right> < 1 - 1 = 0$.  Then $\xb^{\mb_1 + \mb_2} \notin R$ and thus $\xb^{\mb_1} \notin \omega_R^{-1}$, giving us a contradiction.
\end{proof}

In the case that $d = n$ and so $\sigma$ is simplicial, we note that $\tau$ and $\tilde{\tau}$ are actually cones, and they will be isomorphic to $\sigma$ with a new cone point.

In the case of Hibi rings, we could characterize when $\tr(\omega_R) = \m^\ell$ for some $\ell \geq 0$.  In the general toric case, this characterization no longer holds.  For example, in Corollary \ref{puncturedhibi}, we saw that if $R$ cannot be written (nontrivially) as a Segre product of smaller Hibi rings, then $R$ is Gorenstein if and only if it is Gorenstein on the punctured spectrum.  We note, however, that the same result does not hold for general toric rings.

\begin{Example}\label{not a Segre product}{\em
Let $R = k[x,xy,xy^2,xy^3]$ be the toric ring given by the cone $\sigma$ drawn below.  Then $R$ cannot be written as a Segre product (except trivially as $K[z] \sharp R$), and $R$ is Gorenstein on the punctured spectrum but not Gorenstein.  Indeed, we have $\omega_R = (xy,xy^2)$ so that $R$ is not Gorenstein, but $\omega_R^{-1} = (y,1,y^{-1})$ so that $\tr(\omega_R) = \m$ and $R$ is Gorenstein on the punctured spectrum (in fact, $R$ is nearly Gorenstein).

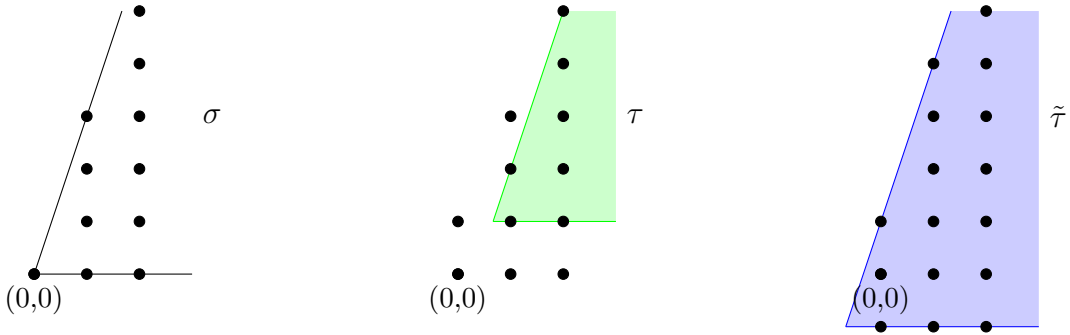
\begin{multicols}{3}

\begin{center}
\begin{tikzpicture}[scale=0.70]
\draw (0,0) -- (3,0);
\draw (0,0) -- (5/3,5);
\draw[fill] (0,0) circle [radius=0.1];
\draw[fill] (2,0) circle [radius=0.1];
\draw[fill] (1,0) circle [radius=0.1];
\draw[fill] (0,0) circle [radius=0.1];
\draw[fill] (1,1) circle [radius=0.1];
\draw[fill] (1,2) circle [radius=0.1];
\draw[fill] (2,1) circle [radius=0.1];
\draw[fill] (2,2) circle [radius=0.1];
\draw[fill] (2,3) circle [radius=0.1];
\draw[fill] (2,4) circle [radius=0.1];
\draw[fill] (2,5) circle [radius=0.1];
\draw[fill] (1,3) circle [radius=0.1];
\node [right] at (3,3) {$\sigma$};
\node [below] at (0,0) {\small{(0,0)}};
\end{tikzpicture}
\end{center}

\columnbreak

\begin{center}
\begin{tikzpicture}[scale=0.70]
\fill[green, opacity=.2] (3,5) -- (2,5) -- (2/3,1) -- (3,1);
\draw[green] (2/3,1) -- (3,1);
\draw[green] (2/3,1) -- (2,5);
\draw[fill] (0,0) circle [radius=0.1];
\draw[fill] (2,0) circle [radius=0.1];
\draw[fill] (1,0) circle [radius=0.1];
\draw[fill] (0,0) circle [radius=0.1];
\draw[fill] (1,1) circle [radius=0.1];
\draw[fill] (1,2) circle [radius=0.1];
\draw[fill] (2,1) circle [radius=0.1];
\draw[fill] (2,2) circle [radius=0.1];
\draw[fill] (2,3) circle [radius=0.1];
\draw[fill] (2,4) circle [radius=0.1];
\draw[fill] (2,5) circle [radius=0.1];
\draw[fill] (1,3) circle [radius=0.1];
\draw[fill] (0,1) circle [radius=0.1];
\node [right] at (3,3) {$\tau$};
\node [below] at (0,0) {\small{(0,0)}};
\end{tikzpicture}
\end{center}

\columnbreak

\begin{center}
\begin{tikzpicture}[scale=0.70]
\fill[blue, opacity=.2] (3,5) -- (4/3,5) -- (-2/3,-1) -- (3,-1);
\draw[blue] (-2/3,-1) -- (3,-1);
\draw[blue] (-2/3,-1) -- (4/3,5);
\draw[fill] (0,0) circle [radius=0.1];
\draw[fill] (2,0) circle [radius=0.1];
\draw[fill] (1,0) circle [radius=0.1];
\draw[fill] (0,0) circle [radius=0.1];
\draw[fill] (1,1) circle [radius=0.1];
\draw[fill] (1,2) circle [radius=0.1];
\draw[fill] (2,1) circle [radius=0.1];
\draw[fill] (2,2) circle [radius=0.1];
\draw[fill] (2,3) circle [radius=0.1];
\draw[fill] (2,4) circle [radius=0.1];
\draw[fill] (2,5) circle [radius=0.1];
\draw[fill] (1,3) circle [radius=0.1];
\draw[fill] (0,-1) circle [radius=0.1];
\draw[fill] (1,-1) circle [radius=0.1];
\draw[fill] (1,4) circle [radius=0.1];
\draw[fill] (2,-1) circle [radius=0.1];
\draw[fill] (0,1) circle [radius=0.1];
\node [right] at (3,3) {$\tilde{\tau}$};
\node [below] at (0,0) {\small{(0,0)}};
\end{tikzpicture}
\end{center}

\end{multicols}
}
\end{Example}

Similarly, the following example shows that in contrast to Hibi rings, the equivalence of (2) and (3) in Theorem \ref{thm:Janet} need not hold for general toric rings.  Namely, the trace of the canonical module of a standard graded toric ring which is Gorenstein on the punctured spectrum need not be a power of the maximal ideal. The following example also shows that a toric ring which is Gorenstein on the punctured spectrum need not be level, as it is the case for Hibi rings, see Corollary~\ref{level}.

\begin{Example}
\label{contrast}
{\em Let $K$ be a field and $R=K[x^3y,x^5y, x^{11}y, x^{23}y]$. Then  $R$ is a $2$-dimensional standard graded Cohen-Macaulay $K$-algebra, and $R\iso S/J$ where $S=K[z_1,z_2,z_3,z_4]$ and $J=(-z_2^4 + z_1^3z_3, -z_3^3 + z_2^2z_4, -z_2^2z_3^2 + z_1^3z_4)$. The ideal  $J$ is toric with the resolution
\[
0 \To R(-5) \dirsum R(-6)\To R(-3)\dirsum  R(-4)^2\To J\To 0,
\]
and the relation matrix
\[
\begin{pmatrix}
-z_2^2 & -z_1^3 \\
z_3 & z_2^2\\
-z_4 & -z_3^2
\end{pmatrix}.
\]
Thus it follows from \cite[Corollary 3.4]{HHSTraceofthecanonical} that $\tr(\omega_R)$ is generated by the residue  classes modulo $J$  of the elements $z_1^3,z_2^2,z_3, z_4$,  and this is not a power of the graded maximal ideal of $R$. Nevertheless it is an ideal of height $2$ in $R$ which shows that $R$ is Gorenstein on the punctured spectrum. Furthermore, we see from the resolution that $R$ is not level.
}
\end{Example}

\smallskip

For simplicial cones $\sigma$ we can use $\tr(\omega_R)$ to give a simple characterization of Gorenstein semigroup rings.  Namely, we recover the following special case of Theorem 6.33 in \cite{BGPolytopesringsandktheory}.

\begin{Proposition}\label{prop:cone}
Let $\sigma= \{\vb \in M_\bbR\:\;  \left<\vb,\ub_i\right> \geq 0, i = 1, \dots, n\}$ where $\ub_i \in N$ are primitive integral vectors.
Then $R = K[S_{\sigma}]$ is Gorenstein if and only if the cone point of $\tau$ is integral, if and only if $U^{-1} \cdot (1, \dots ,1)^\top$ has integral coordinates, where  $U$ is the matrix with rows $\ub_i$.
\end{Proposition}

In fact, this gives another way of showing that Example~\ref{not a Segre product} is not Gorenstein.  Namely, note that the cone point of $\tau$ is $\bb = (2/3,1)$, which is not integral.

\begin{Remark}{\rm
We note that Proposition~\ref{prop:cone} as stated relies on $\sigma$ being simplicial.  If $\sigma$ is not simplicial then $\tau$ and $\tilde{\tau}$ may not be cones.
}\end{Remark}

Similarly, we can classify when simplicial toric rings are Gorenstein on the punctured spectrum.
To state our main result in this direction, we need the following lemma
which follows from an easy computation.  See \cite[Proposition~2.43(a)]{BGPolytopesringsandktheory} for a precise statement.

 \begin{Lemma}\label{lem:Bruns}
Let $\sigma \subset M_\bbR$ be a pointed rational cone with extremal rays through $\ab_1, \dots \ab_n$.
Then $\dim R/(\xb^{\ab_1},\ldots, \xb^{\ab_n})=0$, so that $(\xb^{\ab_1}, \ldots,\xb^{\ab_n})$ is an $\mm$-primary ideal of $R$.
 \end{Lemma}

\begin{Theorem}\label{integralpointsonrays}
Let $R= K[S_\sigma]$, where $\sigma \subset M_\bbR$ is the simplicial cone with extremal rays through $\ab_1, \dots \ab_n \in M$, where we assume $\gcd(\ab_i) = 1$ for all $i$.
Since $\sigma$ is simplicial, $\omega_R$ is defined by a cone $\tau$, with some cone point $\bb =(b_1, \dots, b_n) \in M_\bbR$.  Suppose there are $r$ extremal rays of $\tau$ with integral points. Then
$\height(\tr(\omega_R))\geq r$. Moreover,  $R$ is Gorenstein on the punctured spectrum if and only if there are integral points on every extremal ray of $\tau$.
\end{Theorem}
\begin{proof}
First we note that $R$ is Gorenstein on the punctured spectrum if and only if $\tr(\omega_R) \supset \m^\ell$ for some $\ell \geq 0$ by Lemma 2.1 of \cite{HHSTraceofthecanonical}.  Note that we can write points along the extremal ray of $\tau$ in the direction of $\ab_i$ as $\bb + \ab_it$ for $t \geq 0$.
 Without loss of generality, we may assume that the first $r$ rays of $\tau$ contain some integral point $\pb_i = \bb + \ab_i t_i$ for some $t_i$.
In general, $t_i$ is not an integer, but we can choose an integer  $s_i \geq t_i$  and let
\[\qb_i:= s_i\ab_i-\pb_i =(s_i-t_i)\ab_i-\bb,\]
which has integral coordinates since $s_i\ab_i$ and $\pb_i$ both have integral coordinates.
Then
\[\pb_i + \qb_i=s_i\ab_i.\]
We will show that $\qb_i + \cb \in \sigma$ for all points $\cb \in \tau \cap M$.  Let $\ub_1, \dots , \ub_n \in N$ be the inner normal vectors of $\sigma$ (again, assume $\gcd(\ub_i) = 1$).
Then  $\cb \in \tau$, if and only if $\left<\cb-\bb,\ub_j\right> \geq 0$ for all $j$, so suppose this holds.  We will show that $\left<\qb_i+\cb,\ub_j\right> \geq 0$ for all $j$ so that $\qb_i + \cb \in \sigma$.  Note that:
\begin{align*}
\qb_i+\cb &=(s_i-t_i)\ab_i+(\cb-\bb) \\
{\implies} \left< \qb_i + \cb, \ub_j \right> &= \left< (s_i-t_i)\ab_i, \ub_j \right>  + \left< \cb-\bb, \ub_j\right>.
\end{align*}
Since both terms on the right hand side are positive, we have that $\left<\qb_i+\cb,\ub_j\right> \geq 0$ as desired.  Then in particular, $\xb^{\qb_i+\cb} \in R$ for every $\xb^\cb \in \omega_R$, so that $\xb^{\qb_i} \in \omega_R^{-1}$.  Then since $\xb^{\pb_i} \in \omega_R$, we have $\xb^{\pb_i+\qb_i} =\xb^{s_i\ab_i} \in \tr(\omega_R)$.  In particular, for each $i=1,\ldots,r$, there is some point $s_i\ab_i$ with integral coordinates along the ray $\bb + \ab_i t_i$ such that $\xb^{s_i\ab_i} \in \tr(\omega_R)$. Therefore,
\[
\dim R/\tr(\omega_R)=\dim R/\sqrt{\tr(\omega_R)}=
\dim R/\big(\xb^{\ab_1},\ldots, \xb^{\ab_r}, \sqrt{\tr(\omega_R)}\big)\leq  \dim R/(\xb^{\ab_1},\ldots, \xb^{\ab_r}).
\]

By Lemma~\ref{lem:Bruns} we have that $(\xb^{\ab_1},\ldots, \xb^{\ab_n})$ is an $\mm$-primary ideal of $R$ which implies that the image of the ideal  $(\xb^{\ab_{r+1}},\ldots, \xb^{\ab_n})$ is primary to the maximal ideal of  $R/(\xb^{\ab_1},\ldots, \xb^{\ab_r})$. Therefore, $\dim R/(
\xb^{\ab_1},\ldots, \xb^{\ab_r})\leq n-r$, and hence $\dim R/\tr(\omega_R)\leq n-r$, which implies that $\height(\tr(\omega_R))\geq r$.

In particular,  if on all extremal rays of $\tau$ there exists an integral point, then $\height(\tr(\omega_R))=n$, and $R$ is Gorenstein on the punctured spectrum.

\smallskip

On the other hand, note that if $\vb \in \sigma$ is on the extremal ray $t\ab_i$ for $t \geq 0$, and $\vb = \vb_1 + \vb_2$ for $\xb^{\vb_1} \in \omega_R$, $\xb^{\vb_2} \in \omega^{-1}_R$, then $\vb_1$ must be on the extremal ray $\bb + t\ab_i$ of $\tau$ (and similarly $\vb_2$ must be on the extremal ray $-\bb + t\ab_i$).  In particular, if this extremal ray on $\tau$ has no integral points, we have that $\xb^\vb \notin \tr(\omega_R)$ for any such $\vb$.  Since $\xb^{\ell\vb} \in \m^\ell$ for $\vb$ the first integral point along this extremal ray of $\sigma$,  we have $\tr(\omega_R) \not\supset \m^\ell$ for any integer $\ell$.  In particular, $R$ is not Gorenstein on the punctured spectrum.
\end{proof}

Again, Theorem~\ref{integralpointsonrays} gives us another way to check that Example~\ref{not a Segre product} is Gorenstein on the punctured spectrum.  Namely, we simply observe that both extremal rays of $\tau$ have integral points.

More specifically, we can check when we are in the situation above numerically.  In particular, the following result tells us when an extremal ray of $\tau$ has an integral point.

\begin{Proposition}
\label{integral}
Let $\bb=(b_1,\ldots,b_n)\in \QQ^n$ and $\ab=(a_1,\ldots,a_n)\in \ZZ^n$ a nonzero vector with $\gcd(\ab)=1$, and  let $I=\{i\:\; a_i \neq 0\}$. We may assume that $a_j\neq 0$. Furthermore, let  $c_i=\lceil b_i\rceil-b_i$ and  $e_i= a_ic_j-a_jc_i$ for $i\in I\setminus\{j\}$.  Then there is  an integral point on the   ray $\bb+t\ab$, $(t\geq 0)$  if and only if
\begin{enumerate}
\item[{\em (1)}] $b_i\in \ZZ$ for $i\not\in I$,
\item[{\em (2)}] the numbers $e_i$ with $i\in I \setminus \{j\}$ are integers, and
\item[{\em (3)}] there exists an integer $t_j>0$ such that  $e_i+a_it_j\equiv 0\mod a_j$ for  $i\in I \setminus \{j\}$.
\end{enumerate}
\end{Proposition}

\begin{proof} Since $(\bb+t\ab)_i =b_i$ for $i\not\in I$, the ray $\bb+t\ab$, $(t\geq 0)$ can have an integral point on it only if (1) holds, and we have to only consider the components $a_i$ with $i \in I$. Thus in the following we may as well assume that  $I=[n].$ For simplicity we may further assume that $j=1$.

For $t\geq 0$,  we define  $t_i$ by the equation $t=(1/a_i)(c_i+t_i)$. Then $(\bb+t\ab)_i =\lceil b_i\rceil+t_i$. Thus the $i$th component of $\bb+t\ab$ is an integer if and only if $t_i$ is an integer. Therefore, $\bb+t\ab$ is an integer point  if and only if
\[
t=(1/a_1)(c_1+t_1)=\cdots = (1/a_d)(c_d+t_d) \quad \text{with integers}\quad t_1,\ldots,t_d.
\]
The equations
\[
(1/a_1)(c_1+t_1)= (1/a_i)(c_i+t_i)
\]
give us
\[
a_ic_1-a_1c_i=a_1t_i-a_it_1
\]
for $i=2,\ldots,d$. Thus if $\bb+t\ab$ is an integral point, the right hand terms $a_1t_i-a_it_1$ are integers, and so the left hand terms must be integers as well. This shows that if the ray  $\bb+t\ab$,  $t\geq 0$ has  an integral point, then  for $i=2,\ldots,d$ the numbers $a_ic_1-a_1c_i$ are all integers.

In fact, we have that the  ray $\bb+t\ab$, $t\geq 0$  has  an integral point if and only if (1) holds and there exists a  vector $(t_1,\ldots,t_d)$ with integral coordinates which  is a solution to the equations $e_i=a_1t_i-a_it_1$ $(i=2,\ldots,d)$, where $e_i= a_ic_1-a_1c_i$. Thus the ray has an integral point if and only if (1),  (2) and (3) hold.
\end{proof}

As an immediate consequence of Theorem~\ref{integralpointsonrays} and Proposition~\ref{integral} we obtain the following.

\begin{Corollary}
\label{first condition}
With the assumptions and notation of Theorem~\ref{integralpointsonrays}, the following conditions are equivalent:
\begin{enumerate}
\item[{\em (a)}]  $R$ is  Gorenstein on the punctured spectrum.
\item[{\em (b)}]  For  each $j$, the ray $\bb+t\ab_j$ ($t\geq 0$) satisfies the conditions (1), (2) and (3) of Proposition~\ref{integral}
\end{enumerate}
\end{Corollary}

Under additional assumptions on the ray, Proposition~\ref{integral} can be improved as follows.

\begin{Corollary}
\label{special case}
With the  assumptions and notation of Proposition~\ref{integral}, assume that there exist a nonzero component $a_i$ of $\ab$ such that $a_j$ is invertible module $a_i$ for
$j\neq i$. We set $e_{ij}=a_ic_j-a_jc_i$ for $1\leq i,j\leq n$. Then there exists an integral point on the ray $\bb+t\ab$ with  $t\geq 0$, if and only if $e_{ij}$ is an integer for all $i<j$.
\end{Corollary}

\begin{proof} Since $e_{ij}=a_ic_j-a_jc_i=a_jt_i-a_it_j$, the argument as before shows that the numbers $e_{ij}$ are integers if  $\bb+t\ab$  is an integral point with $t= (1/a_j)(c_j+t_j)$ for $j=1,\ldots,n$.

Conversely, assume that the $e_{ij}$ are integers. We may assume that $a_j$ is invertible  modulo $a_1$ for all $j\geq 2$ and that condition (1) in Proposition~\ref{integral} is satisfied. Since $e_i=-e_{i1}$, condition (2) in Proposition~\ref{integral} is also satisfied. It remains to prove that there exists an integer $t_1$  such that $e_i+a_it_1\equiv 0\mod a_1$.  Let $b_ia_i\equiv 1\mod a_1$ for $i=2,\dots, n$. We let $t_1=-e_2b_2 =e_{21}b_2$.  Then $e_2+a_2t_1\equiv 0\mod a_1$. We claim that we also have  $e_i+a_it_1\equiv 0\mod a_1$ for $i>2$, which is equivalent to saying that $-e_2b_2\equiv -e_ib_i\mod a_1$. This in turn is equivalent  $a_2e_i\equiv a_ie_2\mod a_1$. Indeed, we have
\[
a_2e_i-a_ie_2= a_2(a_ic_1-a_1c_i) -a_i(a_2c_1-a_1c_2)=a_1(a_ic_2-a_1c_j)=-a_1e_{i2}.
\]
Since by assumption $e_{2i}$ is an integer, the desired conclusion follows.
\end{proof}

\begin{Remark}{\em
Consider primitive integral vectors $\ab_1,\ldots,\ab_n$, and let $\sigma$ be the cone whose extremal rays are defined by these vectors. We let $A$ be the matrix whose rows are $\ab_1,\ldots,\ab_n$. We may assume that these vectors are labeled such that $|A|>0$. Let $\bb'_1,\ldots,\bb'_n$ be the column vectors of $|A|A^{-1}$, and let $B$ be the matrix whose row vectors are $\bb_i=\bb'_i/\gcd(\bb'_i)$ for $i=1,\ldots,n$. 
Then the vectors $\bb_i$ are inner normal vectors of $\sigma$, and
$B^{-1}\cdot (1, \dots, 1)^\top$  is the cone point of $\tau$.
}
\end{Remark}

\begin{Example}{\em
Let $\ab_1=(3,1,1)$, $\ab_2=(1,3,1)$ and $\ab_3=(1,1,3)$, and $A$ be the matrix with row vectors $\ab_1,\ab_2,\ab_3$.  Then $|A|=20$ and
\[
|A|A^{-1}= \begin{pmatrix}
8 & -2 & -2 \\
-2 & 8 & -2 \\
-2 & -2& 8
\end{pmatrix}.
\]
Therefore, the vectors $(4,-1,-1), (-1,4, -1), (-1,-1-4)$ are the inner normal vectors of $\sigma$ and the row vectors of $B$. Then
\[
B^{-1}=
\begin{pmatrix}
3/10 & 1/10 & 1/10\\
1/10 & 3/10 & 1/10\\
1/10 & 1/10 & 3/10
\end{pmatrix},
\]
so that $B^{-1}\cdot (1, \dots, 1)^\top=(1/2,1/2,1/2)^\top$, which is the cone point of $\tau$.

Since $(1/2,1/2,1/2)+1/2(3,1,1)=(2,1,1)$, we see that the extremal ray has an integral point. The same holds true for the other extremal rays of $\tau$. Of course this could have also been checked by  applying Corollary~\ref{special case}.

Now Proposition~\ref{prop:cone} and Theorem~\ref{integralpointsonrays} imply   that $K[S_\sigma]$ is not Gorenstein, but Gorenstein on the punctured spectrum. }
\end{Example}

 In the following proposition, we provide a necessary condition for having integral points on the extremal rays of simplicial cones. This is a weaker result, which nonetheless has the benefit that we can check the condition only using the vectors $\ub_i$.
\begin{Proposition}
\label{thm:integral points}
Let $R= K[S_\sigma]$, where $\sigma$ is the simplicial cone $\sigma =\{\vb \in M_\bbR\:\;  \left<\vb,\ub_i\right> \geq 0, i = 1, \dots n\}$
for some primitive integral $\ub_i \in N$.  Let $U$ be the $n\times n$ matrix whose $i^{\rm th}$ row is $\ub_i$ for $i=1,\ldots,n$. We will denote by $\mathcal{I}_{k,n}$ the set of subsets of $\{1,\dots, n\}$ of size $k$.  For $I,J \in \mathcal{I}_{k,n}$, let $U_{I,J}$ be the submatrix of $U$ consisting of those entries of $U$ with row indices in $I$ and with column indices in $J$.  Then the rays on the boundary of the cone defining $\omega_R$ have integral points only if for every $I,J, L \in \mathcal{I}_{n-1,n}, J \neq L$, we have
\begin{align*}
\gcd(|U_{I,J}|, |U_{I,L}|) &\;  \Big| \; \sum_{k = 1}^{n-1} (-1)^{\ell+k} |U_{I -\{i_k\}, J- \{j_\ell\}}|
\end{align*}
where $L = [n] \backslash \{\ell\}$, $I = \{i_1, \dots,  i_{n-1}\}$ and $J = \{j_1, \dots, j_{n-1}\}$ (with $i_k < i_{k +1}$ and $j_k < j_{k +1}$ for all $k$).
In particular, $R$ is not Gorenstein on the punctured spectrum if any of the conditions above fail.
\end{Proposition}
\begin{proof}
By Proposition~\ref{descriptionoftautilde}, $\xb^\mb \in \omega_R$ if and only if $\left<\mb, \ub_i\right> \geq 1$ for all $i \in [n]$.  In particular, the extremal rays of this cone are given by $\{\vb \:\;  \left<\vb, \ub_i\right> =1, \text{ for } i \in I \text{ and } I\in \mathcal{I}_{n-1,n}\}$.
Consider $J,I \in \mathcal{I}_{n-1,n}$.  We note that if both $|U_{I,J}|$ and $|U_{I,L}| = 0$, then the condition above will hold trivially.  Further, since the vectors $\ub_i$ give a cone, we have that $|U_{I,J}| \neq 0$ for some choices of $I$,$J$, so we will assume $|U_{I,J}| \neq 0$.  Write $I = \{i_1, \dots i_{n-1}\}$ and $J = \{j_1, \dots j_{n-1}\}$ (with $i_k < i_{k +1}$ and $j_k < j_{k +1}$).  Then $U_{I,J}$ is invertible, with inverse $(1/|U_{I,J}|)B$, where $B$ is the cofactor matrix of $U_{I,J}^\top$, namely:
$$
B_{\ell,k} := (-1)^{\ell+k} |U_{I \backslash\{i_k\},J \backslash\{j_\ell\}}|
$$
Say $J = [n] \backslash \{j\}$.  We know that $\xb = (x_1, \dots, x_n)$ is on the extremal ray $\{\vb \:\; \left<\ub_i,\vb\right> = 1, \text{ for } i \in I\}$ if and only if:
\begin{equation*}
U_{I,[n]} \cdot \xb = (1, \dots, 1)^\top
\end{equation*}
i.e. if and only if
\begin{equation*}
B \cdot U_{I,[n]} \cdot \xb = B \cdot (1, \dots, 1)^\top
\end{equation*}
In particular, from the $\ell^{\rm th}$ row of the above, we get:
\begin{align*}
|U_{I,J}|x_\ell \pm |U_{I,L}| x_j &= \sum_{k = 1}^{n-1} B_{\ell,k} \\
\text{ or } |U_{I,J}|x_\ell  \pm |U_{I,L}| x_j &= \sum_{k = 1}^{n-1} (-1)^{\ell+k} |U_{I \backslash\{i_k\}, J\backslash\{j_\ell\}}|
\end{align*}
where $L = [n] \backslash \{\ell\}$, and the sign depends on $\ell,j$.  Thus, our condition above must hold in order to have integral solutions $x_\ell,x_j$ to this equation, and so the condition above must hold in order for $R$ to be Gorenstein on the punctured spectrum by Proposition~\ref{thm:integral points}.
\end{proof}

In the 3-dimensional case, this simplifies to the following:
\begin{Corollary}\label{3dimcase}
$R = K[S_{\sigma}]$ is not Gorenstein on the punctured spectrum if for any $I,J,L \in \mathcal{I}_{2,3}$ with $J \neq L$ and $\{\ell\} = J \cap L$, we have that
\begin{equation*}
\gcd(|U_{I,J}|, |U_{I,L}|) \not\Big|   \left|\sum_{i \in I} (-1)^{i + \ell} u_{i,\ell}\right|.
\end{equation*}
\end{Corollary}

\begin{Example}\label{exam:Janet}{\em
Let $R$ be the toric ring given by the cone $\sigma= \{ \vb\:\;  \left<\vb,\ub_i\right> \geq 0, i=1,2,3\}$, where  $\ub_1 = (1,-2,2), \ub_2 =(-2,1,0)$ and $\ub_3 = (3,-1,-4)$. Note that the rays of the cone $\sigma$ go through $(4, 8, 1), (2, 2, 1)$, and  $(2, 4, 3)$.  Then $\tr(\omega_R)$ cannot contain any power of the maximal ideal.
By Corollary \ref{3dimcase}, it suffices to check that letting $I = \{1,3\}$, $J = \{1,2\}$, $L= \{2,3\}$, and $\{\ell\} = \{2\} = J \cap L$, we have that
\begin{equation*}
\gcd(|U_{I,J}|, |U_{I,L}|) \not\Big|   \left|\sum_{i \in I} (-1)^{i + \ell} u_{i,\ell}\right|.
\end{equation*}
\vspace{-.2cm}
Namely, note that:
\begin{equation*}
{\rm gcd}(|U_{\{1,3\},\{1,2\}}|, |U_{\{1,3\},\{2,3\}}|) = {\rm gcd}(-5, 10) = 5 \not\Big|   \left| - u_{1,2} + u_{3,2} \right| = -1.
\end{equation*}
More specifically, note that $\xb=(x_1,x_2,x_3)$ is on the ray given by $\ub_i$ for $i \in I$ if and only if:
\begin{equation*}
\begin{pmatrix}
3 & -1 & -4 \\
1 & -2 & 2
\end{pmatrix}
\cdot
\begin{pmatrix}
x_1 \\
x_2 \\
x_3
\end{pmatrix}
=
\begin{pmatrix}
1 \\
1
\end{pmatrix}
\end{equation*}
if and only if
\vspace{-2mm}
\begin{equation*}
\begin{pmatrix}
-2 & 1 \\
-1 & 3
\end{pmatrix}
\cdot
\begin{pmatrix}
3 & -1 & -4 \\
1 & -2 & 2
\end{pmatrix}
\cdot
\begin{pmatrix}
x_1 \\
x_2 \\
x_3
\end{pmatrix}
=
\begin{pmatrix}
-2 & 1 \\
-1 & 3
\end{pmatrix}
\cdot
\begin{pmatrix}
1 \\
1
\end{pmatrix}
\end{equation*}
or
\vspace{-1mm}
\begin{equation*}
\begin{matrix}
-5x_1 & - & 10x_3  &= -1,  & \text{ and} \\
-5x_2 & + & 10x_3 &=\; \; 2. &
\end{matrix}
\end{equation*}
Note that there are no integral solutions to these equations, so that there is no integral point along this extremal ray which can also be written as $(-8/5, -11/5, -9/10) + t(2,2,1)$ and so $R$ is not Gorenstein on the punctured spectrum
}
\end{Example}

\noindent{\bf Acknowledgement.} {FM was partially supported by EPSRC grant EP/R023379/1.}

\bibliographystyle{alpha}
\bibliography{Gor}{}

\end{document}